\documentclass{birkmult} 
\input xypic

\usepackage{amsmath, amscd, amssymb,}

\theoremstyle{plain}
  \newtheorem{theorem}{Theorem}

  \newtheorem{proposition}{Proposition}
 \newtheorem{observation}{Observation}
  \newtheorem*{conjecture*}{Conjecture}

\theoremstyle{definition}
  \newtheorem{definition}{Definition}
  
  \newtheorem*{convention*}{Convention}
  
\theoremstyle{remark}

\renewcommand{\theenumi}{\roman{enumi}}
\renewcommand{\labelenumi}{(\theenumi)}

\begin{document}

\title [Cyclic theory for CDGA's and string cohomology , ] 
      {Cyclic theory for commutative differential graded algebras and s--cohomology. }

\author{Dan Burghelea}


\address{
         Dept. of Mathematics, 
         The Ohio State University, 
         231 West Avenue, Columbus, OH 43210, USA.}

\email{burghele@mps.ohio-state.edu}

\keywords{}


\date{\today}
\begin{abstract} 
In this paper one considers  three homotopy functors on the category of manifolds , $hH^\ast, cH^\ast, sH^\ast,$ and parallel them with other three homotopy functors on the category of connected commutative differential graded algebras, $HH^\ast, CH^\ast, SH^\ast.$ If $P$ is a smooth 1-connected manifold and the algebra is the de-Rham algebra of $P$ the two pairs of functors agree but in general do not.
The functors $ HH^\ast $ and $CH^\ast$   can be also derived as Hochschild resp. cyclic homology of 
commutative differential graded algebra, but this is not the way they are introduced here. The third $SH^\ast ,$ although inspired from negative cyclic homology,  can not be identified with any sort of cyclic homology of any algebra. The functor $sH^\ast$ might play some role in topology. Important tools in the construction of the functors $HH^\ast, CH^\ast $and $SH^\ast ,$ in addition to the linear algebra suggested by cyclic theory, are  Sullivan minimal model theorem and the "free loop" construction 
described in this paper. 
\end{abstract}

\maketitle

\setcounter{tocdepth}{1}

\hskip .8 in {(dedicated to A. Connes for his 60-th birthday)}

\tableofcontents
 \section{Introduction}
  \label{S:intro}

This paper deals with commutative  graded  algebras and the results are of significance in  
``commutative" geometry/topology. However they were inspired largely by the  linear algebra underlying Connes' cyclic theory.  The topological results  formulated here, 
Theorem \ref{T2} and Theorem \ref{T3},  were first established as a consequence of the identification of the cohomology 
resp. $S^1-$equivariant cohomology  of the free loop spaces of a 1-connected smooth manifold with the Hochschild resp. cyclic homology of its de-Rham algebra, cf. \cite {J}, \cite{BFG}, \cite{B2}.  In this paper  this identification is circumvented. Still,  the results   illustrate the powerful influence of Connes'  mathematics in areas outside the non commutative geometry. 

In this paper, inspired by the relationship between Hochschild, cyclic and negative cyclic homology of a unital algebra, one considers two systems of graded vector space valued homotopy functors $hH^\ast, cH^\ast, sH^\ast $and $HH^\ast, CH^\ast, SH^\ast$  and investigate their relationship. 
The first  three functors are defined on the category of smooth manifolds and smooth maps via the free loop space $P^{S^1}$ of a smooth manifold $P,$ which is a smooth $S^1$-manifold of infinite dimension.   The next three functors are  
 defined on the category of connected commutative differential graded algebras via an algebraic analogue of ``free loop"  construction  and via  Sullivan minimal model  theorem, Theorem  \ref{T1}. The relationship between them is suggested  by the general nonsense- diagram Fig 2 in section \ref{SMC}.

When applied to the De--Rham algebra of a 1-connected smooth manifold the last three functors take the same values as the first three. 
This is not the case when the smooth manifold is not 1-connected; the exact relationship will be addressed in a future work.

The first three functors are based on a formalism  (manipulation with differential forms) which can be considered for any smooth  (finite or infinite dimensional) manifold  $M$ and  any smooth vector field $X$ on $M.$  However it seems to be  of relevance when the vector field $X$ comes from a smooth $S^1-$action on $M.$ This is  of  mild  interest if the manifold is of finite dimension but  more  interesting  when the manifold is of infinite dimension. In particular,  it is quite interesting  when $M= P^{S^1},$  the free loop space of $P,$ and the  action is the canonical $S^1-$action on $P^{S^1}.$
Manipulation with differential forms on $P^{S^1}$  leads to the graded vector spaces   $hH^\ast (P)$, 
$cH^\ast (P),$ $sH^\ast(P) $ with the first two  being the cohomology resp. the $S^1-$equivariant cohomology of $P^{S^1}$ but  $sH^\ast$ being  a new homotopy functor, referred here as  s--cohomology. 

This functor was first  introduced in \cite {B1}, \cite{B2} but so far not 
seriously investigated. 
The functor $sH^\ast$ relates, at least in the case of a 1-connected manifold $P$, the Waldhaussen algebraic $K-$theory  of $P$ and the Atiyah--Hirtzebruch complex $K-$theory  (based on complex vector bundles) of $P.$  It has  a rather simple description in terms of infinite sequence of smooth invariant 
differential forms on $P^{S^1}.$ 

The additional structures on $P^{S^1},$  the power maps $\psi_k$  $k=1,2,
\cdots,$ and the involution $\tau= \psi_{-1},$  provide endomorphisms of
$hH^\ast (P)$, $cH^\ast (P), sH^\ast(P) $  whose eigenvalues and eigenspaces are interesting
issues.  They are  clarified only when $P$ is 1-connected. This  is done in view of the relationship  with the functors $HH^\ast, CH^\ast, SH^\ast.$

 It might be only a coincidence but still a  appealing observation that the symmetric resp. 
 antisymmetric part of  $sH^\ast (P)$ w.r. to the canonical involution $\tau$ calculates,  for a 1-connected manifold  $P$ and  in the stability   range,  the vector space  $Hom (\pi_\ast(H/\textit{Diff} (P), \ \kappa),$ $\kappa=\mathbb R, \mathbb C;$  the symmetric part when  $\dim P$  is even the antisymmetric part when 
 $\dim P$ is odd, cf.  \cite {Bu}, \cite{B1}.  Here $H/\textit{Diff} (P)$ denotes   the (homotopy) quotient  of the space of homotopy equivalences of $P$ by the group of diffeomorphisms with the $C^\infty-$topology.

The functors $HH^\ast , CH^\ast,  SH^\ast$ are the algebraic version of $hH^\ast, cH^\ast, sH^\ast$ and are defined on the category of (homologically connected) commutative differential graded algebras. Their definition uses  the  ``free loop'' construction, an algebraic analogue of the free loop space,  described in this paper only for free connected commutative differential graded algebras $(\Lambda[V], d_V).$ 
A priori these functors are   defined  only for free connected commutative differential graded algebras.
Since they are homotopy functors they  extend to all connected commutative differential graded algebras  via  Sullivan minimal model theorem, Theorem \ref{T1}. 

Using the definition presented here one can take full advantage of
 the  simple form  the algebraic analogue of the power maps  take on the free loop construction. As a consequence one obtains a simple description of  the eigenvalues and eigenspaces of the endomorphisms induced from the algebraic power maps on $HH^\ast$ and $CH^\ast$ and implicitly understand their additional structure. 

The extension of the results of Sullivan--Vigu\'e, cf. \cite{VS}, to incorporate  $S^1-$action and the  power maps in the minimal model of $P^{S^1}$ summarized in Section \ref{S:minimal}  leads finally to results about $hH^\ast, cH^\ast, sH^\ast$  when $P$ is 1-connected, 
cf. Theorem \ref{T3}.

In addition to the algebraic definition  of $HH^\ast, CH^\ast, SH^\ast $  this paper contains the proof of the  homotopy invariance of $sH^\ast.$ 
\vskip .2in

\section {Mixed complexes, a formalism inspired from Connes' cyclic theory}
\label {SMC}
A mixed complex $(C^\ast ,  \delta^\ast,  \beta_\ast)$ consists of a   graded vector space $C^\ast$  ( $\ast$ a non negative integer)  and  linear maps,  $\delta^\ast  : C^\ast \to C^{\ast+1},$ 
$\beta_{\ast+1}: C^{\ast+1}\to C^\ast $ 
which satisfy: 
\begin{equation*}
\begin{aligned}
\delta^{\ast+1}\cdot \delta^{\ast}=&0\\
\beta_{\ast}\cdot \beta_{\ast+1}=&0\\
\beta_{\ast+1} \cdot \delta^\ast + \delta^{\ast-1} \cdot \beta_{\ast}=&0.
\end{aligned}
\end{equation*} 

When  
no risk of confusion 
the index $\ast $ will be deleted and we write $ (C, \delta, \beta) $ instead of $ (C^\ast, \delta^\ast, \beta_\ast) $. 
Using the  terminology of  \cite{B2}, \cite{BV} a mixed complex can be viewed either as  a cochain complex $(C^\ast, d^\ast)$ equipped with an $S^1-$action $\beta_\ast$, or as a chain complex
 $(C^\ast, \beta_\ast)$ equipped with an algebraic $S^1-$action $ \delta^\ast .$

To a mixed complex $(C^\ast, \delta^\ast, \beta_\ast)$ one associates a number of cochain, chain and $2-$periodic  cochain complexes, and then their cohomologies, homologies  and $2-$periodic  cohomologies\footnote{ We will use the word "homology" for a functor derived from a chain complex and ``cohomology" for one derived from a cochain complex. The  $2-$periodic chain and cochain complexes can be identified.}, as follows.

First denote by 
\begin {equation}
\begin{aligned}
^+ C^r= \prod_{ k\geq 0} C^{r-2k}  \ \ & \ \
^- C^r:= \prod_{k\geq 0} C^{r+2k}\\
   \mathbb PC^{2r+1}= \prod_{k\geq 0} C^{2k+1} \ \ & \ \ 
  \mathbb PC^{2r}= \prod_{k\geq 0} C^{2k} \ \text {for any}\  r\\
   PC^{2r+1}= \bigoplus_{k\geq 0} C^{2k+1} \ \ & \ \
    PC^{2r}= \bigoplus_{k\geq 0} C^{2k} \text {for any}\  r.
\end{aligned}
\end{equation}  
 Since our  vector spaces are  $\mathbb Z_{\geq 0}$-graded  the direct product $^+C^r$ involves only finitely many factors. 

Next introduce 
\begin {equation}
\begin{aligned}
^+ D^r_\beta  (w_r, w_{r-2}, \cdots) := & (\delta\omega_r,  (\delta \omega_{r-2} +\beta\omega_r ),\cdots)\\
^+ D^\delta_r (w_r, w_{r-2}, \cdots) := & ((\beta \omega_r+ \delta \omega_{r-2 }), (\beta \omega_{r-2}+ \delta \omega_{r-4 }),\cdots )\\
^- D^r_\beta  (\cdots, w_{r+2}, w_{r}) := & (\cdots, (\beta \omega_{r+4} +\delta \omega_{r+2}), (\beta \omega_{r+2} +\delta \omega_r ))\\ 
^- D^\delta_r (\cdots, w_{r+2}, w_{r}) := & ( \cdots , ( \delta \omega_r + \beta \omega_{r+2 }), \beta \omega_r ) \\
D^{2r}(\cdots, \omega_{2r+2}, \omega_{2r}, \cdots \omega_0)=& (\cdots, (\delta \omega_{2r} + \beta \omega_{2r+2}),\cdots)     \\
D^{2r+1}(\cdots, \omega_{2r+3}, \omega_{2r+1}\cdots\omega_1)= &(\cdots, (\delta \omega_{2k+1}+ \beta \omega_{2k+3}), \cdots).     
\end{aligned}
\end{equation}  
Finally consider  the cochain complexes
$$\mathcal C:= (C^\ast, \delta^\ast), \ \  ^+ \mathcal C_\beta:= (^+C^\ast, ^+D^\ast_\beta),  \ \ ^-\mathcal C_\beta:= (^- C^\ast, ^- D^\ast_\beta),$$
the chain complexes
$$\mathcal H:= (C^\ast, \beta_\ast),\ \ ^+ \mathcal H^\delta:= (^+C^\ast, ^+D^\delta_\ast),  \ \ ^-\mathcal H^\delta:= (^- C\ast, ^- D^\delta_\ast)$$
and the  $2-$periodic cochain complexes \footnote{ here $(\mathbb PC^\ast, D^\ast)$ is regarded as a cochain complex with  $D^\ast$ obtained  from  the degree $+1$  derivation $\delta$ perturbed by the degree $-1$ derivation $\beta$ ; the same complex can be regarded chain complex with  $D^\ast$ obtained  from  the degree $-1$ derivation $\beta$ perturbed by the degree $+1$ derivation $\delta$;
the cohomology for the first is the same as homology for the second }
$$P\mathcal C:= (PC^\ast, D^\ast) ,  \ \  \mathbb P\mathcal C:= (\mathbb PC^\ast, D^\ast)
$$ 
whose cohomology, homology, $2-$periodic  cohomology  are denoted by 

\begin{equation*}
\begin{aligned}
H^\ast: = H^\ast(C,\delta) , \  \   \ & ^+H^\ast_\beta: = ^+H^\ast_\beta (C,\delta, \beta) , 
\ &^-H^\ast_\beta&:=  ^-H^\ast_\beta (C,\delta, \beta), \\
H_\ast := H_\ast (C, \beta) , \  \   \ &    ^ +H^\delta_\ast :=  ^+H^\delta_\ast (C,\delta, \beta),            
\ &^-H^\delta_\ast&:= ^-H^\delta_\ast (C,\delta, \beta),  \\
PH^\ast:= PH^\ast (C,\delta, &\beta) ,   \  &\mathbb PH ^\ast&:= \mathbb PH^\ast (C,\delta, \beta). 
\end{aligned}
\end{equation*}
In this paper the chain complexes $\mathcal H, ^\pm {\mathcal H}^\delta, $ will only be used to derive conclusions about the cochain complexes  $\mathcal C,   ^\pm {\mathcal C}_\beta, \ \mathbb P \mathcal C.$

 The obvious inclusions and projections lead to the  following commutative 
diagrams of short exact sequences 
\vskip .1in

\diagram 
&&&0 \rto & ^- \mathcal H^\delta_{\ast }\rto \dto &\mathbb P \mathcal C^\ast \rto \dto & ^+\mathcal H^\delta_{\ast-2} \rto,\dto^{id} &0\\
&&&0 \rto & \mathcal H_{\ast} \rto         & ^+\mathcal H^\delta_{\ast } \rto        & ^+\mathcal H^\delta_{\ast-2} \rto&0
\enddiagram

\vskip .2in

\diagram 
&&&0 \rto & ^+ \mathcal C^{\ast-2}_\beta  \rto^{i^{\ast-2}}\dto^{id} & ^+\mathcal C^\ast_\beta\rto\dto^{I^{\ast}} & \mathcal C^\ast\rto \dto &0\\
&&& 0 \rto &^+ \mathcal C^{\ast -2}_\beta \rto^{I^{\ast-2}} &\mathbb P \mathcal C^\ast \rto & ^-\mathcal C^\ast_\beta \rto &0.
\enddiagram

They  give 
 rise to the following commutative diagram of long exact sequences 
\vskip .1in

\diagram
\cdots \rto&^-H_{r}^\delta\rto\dto^{h_r} & \mathbb PH(r) \rto \dto & ^+H_{r-2}^\delta\rto\dto^{id}                 & ^-H_{r-1}^\delta   \rto \dto^{h_{r-1}}  &  \cdots\\
\cdots \rto&H_{r} \rto^{J_r}        & ^+H_{r}^\delta\rto^{S_r }                           & ^+H_{r-2}^\delta\rto^{B_{r-2}}                        & H _{r-1} \rto  &\cdots 
\enddiagram
\vskip .1in
\hskip 2in Fig 1.

\vskip .2in
\diagram
\cdots \rto&^+H^{r-2}_\beta\rto^{S^{r-2}}\dto^{id} & ^+H^{r}_\beta \rto \dto^{\mathbb I^\ast} & H^{r}\rto^{B^r}\dto^{i^r}                 & ^+H_\beta^{r-1}   \rto \dto^{id}  &  \cdots\\
\cdots \rto&^+H^{r-2}_\beta\rto ^{\mathbb I^{r-2}}       & \mathbb PH^r\rto ^{\mathbb J^r}                           & ^-H^{r}_\beta\rto ^{\mathbb B^r}                       & ^+H^{r-1}_\beta  \rto  &\cdots 
\enddiagram
\vskip .1in
\hskip 2in Fig 2.

and 

\diagram & ^+H^{r-2}_\beta\dto^{\mathbb I^{r-2}} \rto^{S^{r-2}} &^+H^{r}_\beta\dlto^{\mathbb I^r}\dto \\
&\mathbb PH^r = \mathbb PH^{r+2} &  
{\underset {\underset S{\rightarrow}} \lim 
 ^+H^{r+ 2k}_\beta}\lto .  
\enddiagram 

\noindent The diagram (Fig1) is the one familiar in the homological algebra of  Hochschild versus
cyclic homologies, cf \cite{Lo}. The diagram  Fig 2 is the one we will use in this paper. 

Note that Hochschild, cyclic, periodic cyclic, negative cyclic homology of an associative unital algebra $A$ as defined in \cite{Lo},   is $H_\ast, ^+H^\delta_\ast, $ $\mathbb PH^\ast, ^-H^\beta_\ast$ of the Hochschild mixed complex  with $C^r: = A^{\otimes (r+1)}$ , $ \beta$ the Hochschild boundary,  and $\delta^r= (1-\tau_{r+1})\cdot s_r\cdot (1 +\tau_r + \cdots \tau^r_r)$ where $\tau_r( a_0\otimes a_1\otimes \cdots a_r)=    
(a_r\otimes a_0\otimes \cdots a_{r-1})$ and $s_r (a_0\otimes a_1\otimes \cdots a_r)= (1\otimes a_0\otimes a_1\otimes \cdots a_r).$

\vskip .1in

A morphism $f: (C^\ast_1,  \delta^\ast _1,\beta^1_\ast) \to (C^\ast _2,  \delta^\ast_2, \beta^2_\ast)$
 is a degree preserving linear map which intertwines $\delta'$s and $\beta'$s.
 It induces  degree  preserving linear maps between any of the homologies /cohomologies  defined above.
The following elementary observations  will be used below.
\begin{proposition}\label{P1}
Let  $(C, \delta, \beta)$ be a mixed cochain complex.

1.$ PH^r = \underset {\underset S{\rightarrow}} \lim ^+H^{r+2k}_\beta,$  where $S^{k+2r}: ^+H^{k+2r}_\beta
\to ^+H^{k+2r+2}_\beta$ is induced by the inclusion 
$^+ \mathcal C^{\ast}_\beta  \to  ^+\mathcal C^{\ast+2}_\beta.$ 

2. The following is an exact sequence 

\diagram 
&0 \rto &\underset {\underset S {\leftarrow}} { \lim'} \ ^+H^\delta _{r-1 +2k} \rto &\mathbb P H^r \rto, & \underset{\underset S {\leftarrow}}\lim \ ^+H^\delta_{r+2k}\rto& 0,
\enddiagram  with $S_{k+2r}: ^+H_{k+2r}^\delta
\to ^+H_{k+2r-2}^\delta$  induced by the projection $^+\mathcal H^\delta_\ast\to ^+\mathcal H^\delta_{\ast-2}.$ 

Let $f^\ast:(C^\ast_1, \delta^\ast _1, \beta^1_\ast) \to (C^\ast_2, \delta^\ast _2, \beta^2_\ast)$ be a morphism of mixed complexes.

3. If $H^\ast(f)$  is an isomorphism then so is $^+H^\ast_\beta (f)$ and $PH^\ast (f).$

4. If $H_\ast (f)$ is an isomorphism then so is $^+H^\delta_\ast (f)$ and $\mathbb PH^\ast(f).$

5. If $H^\ast(f)$ and $H_\ast(f)$ are both isomorphisms then, in addition to the conclusions in (3) and (4), $^-H^\ast_\beta(f)$ is an isomorphism.
\end{proposition}

\begin{proof}

(1): Recall that  a  direct  sequence of  cochain complexes 

$$\mathcal C^\ast_0\overset {i_0} \rightarrow \mathcal C^\ast_1 \overset {i_1} \rightarrow \mathcal C^\ast_2 \overset {i_2} \rightarrow \cdots$$
induces, by passing to cohomology,  the direct sequence 
$$H^\ast (\mathcal C^\ast_0)\overset {H(i_0)} \rightarrow H^\ast(\mathcal C^\ast_1) \overset {H(i_1)} \rightarrow H^\ast(\mathcal C^\ast_2) \overset {i_2} \rightarrow \cdots$$
 and that $H^j (\underset{\rightarrow} \lim \mathcal C^{\ast}_i)= \underset{\rightarrow}\lim H^j (\mathcal C^\ast_i)$ for any $j.$ 

(2): Recall  that an inverse  sequence of  chain complexes
$$\mathcal H_\ast^0\overset {p_0} \leftarrow \mathcal H_\ast^1 \overset {p_1} \leftarrow \mathcal H_\ast^2 \overset {p_2} \leftarrow \cdots$$
induces, by passing to homology,  the sequence 
$$H_\ast (\mathcal H^\ast_0)\overset {H(p_0)} \leftarrow H_\ast(\mathcal H^\ast_1) \overset {H(p_1)} \leftarrow H_\ast(\mathcal H^\ast_2) \overset {p_2} \leftarrow \cdots$$
 and the following short exact sequence cf.\cite {Lo}   5.1.9.

\diagram 
&0 \rto &\underset {\leftarrow} { \lim' } H _{j-1} (\mathcal H^i_\ast) \rto & H_j ( \underset {\leftarrow} \lim \mathcal H^i_\ast) \rto& \underset {\leftarrow} \lim  H _j (\mathcal H^i_\ast)\rto &0 
\enddiagram 
\noindent for any $j.$ 

Item (3) follows by induction on  degree from the naturality of the first exact sequence in the diagram Fig 2 and (1). 

Item (4) follows by induction from the naturality of the  second exact sequence of the  diagram Fig 1 and from (2).

Item (5) follows from the naturality of the second exact sequence in diagram Fig 2 and from (3) and (4).
\end{proof}

The mixed complex $(C^\ast, \delta^\ast, \beta_\ast)$ is called  $\beta-$acyclic if  $\beta_1$ is surjective and $\ker (\beta_r)= \text{im} (\beta_{r+1}).$
If so 
consider the diagram whose rows are short exact sequences of cochain complexes 

\diagram
0\rto&(\text {Im} (\beta)^\ast , \delta^\ast )\drto ^{j} \rto^ i &(C^\ast , \delta^\ast ) \rto^\beta \drto^{ id} & ((\text {Im} (\beta))^{\ast-1},  \delta ^{\ast-1})\rto &0   \\
0\rto&(^+C^{\ast-2}_\beta, ^+D^{\ast-2}_\beta) \rto^{i^{\ast-2}}&(^+C^\ast_\beta, ^+D^\ast_\beta) \rto&  (C^\ast,\delta^\ast)\rto&0
 \enddiagram 
Each  row induces the long exact sequence in the  diagram below and a simple inspection of boundary map in these long exact sequences permits to construct linear maps $\theta^r$ and to verify that the diagram below is commutative.

\diagram 
H^{r-2}(\text{Im} (\beta),\delta)\dto^{\theta^{r-2}} \rto &H^{r}(\text{Im} (\beta),\delta)\dto^{\theta^{r}} \rto &H^r(C,\delta)\dto^{id} \rto &H^{r-1}(\text{Im} (\beta),\delta)\dto^{\theta^{r-1}} \\
^+ H^{r-2}_{\beta}(C,\delta, \beta)\rto &^+H^{r}_{\beta}(C,\delta, \beta)\rto&H^r(C,\delta)\rto&
 ^+H^{r-1}_{\beta}(C,\delta, \beta)
\enddiagram

As a consequence  one verifies  by induction on degree that the inclusion 
\newline $j: (\text{Im} \beta, \delta)\to (^+ C, ^+ D_\beta) $ induces an isomorphism  
$H^\ast  (\text{Im}(\beta), \delta)\to  ^+H^\ast_{\beta}(C,\delta, \beta).$ 
 \vskip .1in

{\bf Mixed complex with power maps and involution}

A collection of degree zero (degree preserving) linear maps $\Psi_k, k=1,2.\cdots , \tau:= \Psi_{-1}$ which satisfy
\begin{enumerate}
\item $\Psi_k \circ \delta= \delta\circ \Psi_k,$
\item $\Psi_k \circ \beta= k \beta \circ \Psi_k,$
\item $\Psi_k \circ \Psi_r= \Psi_r \circ \Psi_k= \Psi_{kr},  \ \   \Psi_1 =id$
\end{enumerate}
will be referred to as ``power maps and involution", or simpler "power maps", $\Psi_k, k=-1, 1,2, \cdots .$
\footnote {We use the notation $\tau$ for $\Psi_{-1}$ to emphasize  that  is an involution and to suggest  consistency with other familiar involutions in homological algebras and topology.}. They provide  the morphisms of cochain complexes 
\begin{equation*}
\begin{aligned}
\Psi_k: \mathcal C \to& \mathcal  C,\\ 
^\pm \Psi_k: ^\pm \mathcal C_\beta \to&  ^\pm \mathcal C_\beta\\ 
 \mathbb P \Psi_k: \mathbb P\mathcal C\to& \mathbb P \mathcal C
 \end{aligned}
 \end{equation*}
defined as follows 

\begin {equation*}
\begin{aligned}
^+ \Psi_k^r  (w_r, w_{r-2}, \cdots)  =  &(\Psi^r (\omega_r), \frac{1}{k} \Psi^{r-2}(\omega_{r-2}) ,\cdots)\\
^- \Psi_k^r  (\cdots, w_{r+2}, w_{r}) := & (\cdots k\Psi_k^{2r+2}(\omega_{r+2} ), \Psi_k^r(\omega_r))\
\end{aligned}
\end{equation*}

and 

\begin{equation*}
\begin{aligned} 
\mathbb P \Psi_k^{2r}(\cdots,  \omega_{2r+2},  \omega_{2r},  \omega_{r-2}, \cdots ,  \omega_0)= \\
=(\cdots, k\Psi_k^{2r+2}(\omega_{r+2}),  \Psi_k^{2r}(\omega_{2r}), 
\frac{1}{k} \Psi_k^{2r-2} (\omega_{2r-2}), \cdots , \frac{1} {k^r} \Psi_k^0 (\omega_0))\\
\mathbb P\Psi_k^{2r+1}(\cdots,  \omega_{2r+3},  \omega_{2r+1},  \omega_{2r-1},\cdots,  \omega_1)=\\ =(\cdots,  k \Psi_k^{2r+3}(\omega_{2r+3}),  \Psi_k^{2r+1} (\omega_{2r+1}), 
\frac{1}{k} \Psi_k^{2r-1} (\omega_{2r-1}), \cdots \frac{1} {k^r} \Psi_k^1 (\omega_1))      
\end{aligned}
\end{equation*}  

Consequently they 
induce the endomorphisms, 
$$\underline \Psi_k^\ast : H^\ast \to H^\ast $$
$$\underline {^\pm \Psi}_k^\ast: ^\pm H^\ast_\beta \to  ^\pm H^\ast_\beta$$
$$\underline {\mathbb P \Psi^\ast}_k: \mathbb PH^\ast \to \mathbb PH^\ast$$

Note that in diagram (Fig2):

$\mathbb J^\ast,$  $J^\ast$ and the vertical arrows intertwine  the endomorphisms induced by $\Psi_k,$ 

$\mathbb B^\ast$ resp. $B^\ast$ intertwine $k ( ^-\underline\Psi_k)$ resp. $k \underline \Psi_k$ with $^+\underline \Psi_k,$

$\mathbb I^{\ast-2}$ resp. $S^{\ast-2}$ intertwine $ ^+\underline \Psi _k$ with  $k \underline{\mathbb P \Psi}_k$ resp. $k  (^+\underline \Psi_k),$  

The  above elementary linear algebra will be applied to CDGA's in the next sections. 
\vskip .2in

\section {Mixed commutative differential graded algebras}\label{SMA}

Let $\kappa$ be a field of characteristic zero (for example $\mathbb Q, \mathbb R, \mathbb C$).

\begin{definition}\label {D1}
\begin{enumerate}

\item A commutative  graded algebra   abbreviated CGA,   is  an associative unital augmentable graded algebra $\mathcal A^\ast,$ (the augmentation is not part of the data)  which is commutative in the graded sense, i.e.
 \begin{equation*}
a_1\cdot a_2= (-1)^{r_1  r_2}a_2\cdot a_1,\ \ 
a_i\in \mathcal A^{r_i}, i=1,2.
\end{equation*}

 \item An exterior  differential $d^\ast_\mathcal A: \mathcal A^\ast \to \mathcal A^{\ast+1},$  is a degree +1-linear map which satisfies
 
\begin{equation*}
 d(a_1\cdot a_2)= d(a_1)\cdot a_2 +(-1)^{r_1} a_1\cdot d(a_2),  a_1\in \mathcal A^{r_1} ,\ \ 
d^{*+1}_{\mathcal  A} d^{*}_{\mathcal  A}=0.
\end{equation*}  

\item An interior differential $\beta^\mathcal A_\ast : \mathcal A^\ast\to \mathcal A^{\ast-1}$ is a degree -1 linear map which satisfies  

\begin{equation*}
\beta(a_1\cdot a_2)= \beta(a_1)\cdot a_2 +(-1)^{r_1} a_1\cdot \beta(a_2),   a_1\in \mathcal A^{r_1},  \ \ 
\beta_{\ast-1}^{\mathcal  A} \beta_{\ast}^{\mathcal  A}=0.
\end{equation*}

\item The exterior and interior differentials $d^\ast$ and $\beta_\ast$ are compatible if 
\begin{equation*}
d^{\ast-1}\cdot \beta_\ast +\beta_{\ast+1}\cdot d^\ast=0.
\end{equation*}

\item A pair $(\mathcal A^\ast, d^\ast),\ \mathcal A^\ast$ a CGA and   $d^\ast$  exterior differential, is called CDGA  and a triple $(\mathcal A^\ast, d^\ast, \beta_\ast),\   \mathcal A^\ast$ a CGA,  $d^\ast$  exterior differential and $\beta_\ast$ interior differential,
with $d^\ast$ and $\beta_\ast$ compatible,  is called a mixed  CDGA.
\end{enumerate}
\end{definition}

A mixed  CDGA is a mixed cochain complex .

A  degree preserving linear map $f^\ast :\mathcal A^\ast \to \mathcal B^\ast$ is a morphism of CGA's, resp. CDGA's, resp. mixed CDGA's  if is a unit preserving graded algebra homomorphism and   intertwines  $d'$s and $\beta'$s
when the case.

We will consider the categories of CGA's, CDGA's and mixed CDGA's.
In all these three categories there is a canonical tensor product and in the category of CDGA's a well defined concept of homotopy between two morphisms (cf. \cite {Lo}, \cite {Ha}) 
\footnote{ let $k(t,dt)$ 
be the free commutative graded algebra generated by the symbol $t$ of degree zero and $dt$ of degree one, equipped with the differential $d(t)= dt.$   A morphism $F: (\mathcal A, d_\mathcal A) \to (\mathcal B, d_\mathcal B) \otimes_k (k(t,dt), d),$  is called elementary homotopy from $f$ to $g$ , $f,g: (\mathcal A, d _\mathcal A)\to 
(\mathcal  B, d_{\mathcal  B}),$
if 
$\rho_0 \cdot F= f,$  and $\rho_1\cdot F=g$
 where 
 
 \begin{equation*}
\begin{aligned}
\rho_0( a\otimes p(t))= &p(0)a,  \   \rho_0(a\otimes p(t)dt)= 0,\\
\rho_1(a\otimes p(t))= &p(1)a, \   \rho_1(a\otimes p(t)dt)= 0,
\end {aligned} 
\end{equation*}
The homotopy is the equivalence relation generated by elementary homotopy.}.
The category of mixed CDGA's is a subcategory of mixed cochain complexes and all definitions and considerations in section \ref {SMC} can be applied. 

For a (commutative) differential graded algebra  
$ (\mathcal  A^\ast, d^\ast_{\mathcal  A}),$ the graded vector space  
$H^\ast(\mathcal  A^\ast, d^\ast)= Ker (d^\ast)/ {Im (d^{\ast-1})}$
is a commutative graded algebra whose 
multiplication is induced by the multiplication in $\mathcal  A^\ast.$
A morphism $f= f^\ast : (\mathcal  A^\ast, d^\ast_{\mathcal  A})\to (\mathcal  B^\ast, d^\ast_{\mathcal  B})$
induces a degree preserving  linear map, $H^{\ast}(f): H^\ast(\mathcal  A^\ast,d_{\mathcal  A}^\ast)\to H^\ast(\mathcal  B^\ast,d_{\mathcal  B}^\ast),$ which is an algebra homomorphism.

\begin{definition}\label {D2}
 A morphism of CDGA's $f,$  with $H^k(f)$  isomorphism for any $k,$ 
 is called a quasi isomorphism. 

The CDGA $(\mathcal A, d_{\mathcal A})$ is called homologically connected  if  $H^0(\mathcal A, d_{\mathcal A})$
$=\kappa$
and homologically 1-connected if homologically  connected and $H^1(\mathcal A, d_{\mathcal A})=0.$
\end{definition}

The full subcategory of homologically connected CDGA's will be denoted by c--CDGA. For all practical purposes (related to geometry and  topology) it suffices to consider only  c-CDGA' s.

\begin{definition} \label{D3}

1. The CDGA  $(\mathcal A, d)$ is called free if $\mathcal A= \Lambda [V],$ where $V= \sum_{i\geq 0}V^i$ is a graded vector space  and $\Lambda[V]$ denotes the free commutative  graded algebra generated by $V.$  If in addition $V^0=0$  is called free connected commutative differential graded algebra, abbreviated fc-CDGA.

2. The CDGA $(\mathcal A, d)$ is called  minimal  if  it is a fc-CDGA  and 
 in addition  
 
 i. $d(V)\subset \Lambda^+[V]\cdot \Lambda^+[V],$ with $\Lambda^+[V]$ the ideal generated by $V$,
 
 ii. $V^1= \oplus_{\alpha\in I} V_\alpha$ with $I$ a well ordered set and $d(V_\beta)\subset 
 \Lambda[\oplus _{\alpha <\beta} V_\alpha]$ (the set $I$ and its order are not part of the data)
  \end{definition}

\begin{observation}\label {O0}
If $(\Lambda[V], d_V)$ is minimal and 1-connected,  then $V^1=0$ and, for $v\in V^i,$ $d_V (v) $is a linear combination of products of elements $v_j\in V^j$ with $j<i$.  In particular for  $v\in V^2$ one has   $dv=0.$  
\end{observation}

The interest of minimal  algebras comes from the following  result  \cite{L}, \cite{Ha}.
\begin{theorem} \label{T1}
1 (D. Sullivan)

1. A quasi isomorphism between two minimal CDGA's is an isomorphism.

2. For any  homologically connected CDGA,  \ $(\mathcal A, d_{\mathcal A}),$  there exists quasi isomorphisms $\theta: (\Lambda[ V], d_V)\to (\mathcal A, d_{\mathcal A}) $ with $(\Lambda[ V], d_V)$ minimal. Such $\theta$ will be called  minimal model of $(\mathcal A, d_\mathcal A).$

3. Given a morphism $f:(\mathcal A, d_{\mathcal A})\to (\mathcal B, d_{\mathcal B})$  and the minimal models $\theta_A: (\Lambda[V_A], d_{V_A})\to (A,d_A)$  and  $\theta_B: (\Lambda[V_B], d_{V_B})\to (B,d_B),$   there exists morphisms  $f' :  (\Lambda [V_A], d_{V_A}) \to (\Lambda[V_B], d_{V_B})$ such that $f \cdot \theta_A$ and $\theta_B\cdot f'$ are homotopic; moreover   any two such $(f')$'s  are homotopic.
\end{theorem}

We can therefore consider the homotopy  category of c--CDGA's, whose the morphisms are homotopy classes of morphisms  of CDGA's.   By the above theorem the full subcategory of fc-CDGA is a skeleton, and therefore any homotopy functor a priori defined on  fc--CDGA's   admits extensions to homotopy functors defined on the full homotopy category of c--CDGA's and all these extensions are  isomorphic as functors. In particular any statement about a homotopy functor on the category c-CDGA  suffices to be verified for fc--CDGA. 

Precisely for any c--GDGA, \  $(\mathcal A, d_\mathcal A),$ choose a minimal model $\theta_\mathcal A:(\Lambda[V_\mathcal A], d_{V_\mathcal A}) \to (\mathcal A, d_\mathcal A)$ and  for any $f:(\mathcal A, d_\mathcal A)\to (\mathcal B, d_\mathcal B)$ choose a  morphism 
$f':(\Lambda[V_\mathcal A], d_{ V_\mathcal A})\to (\Lambda[V_\mathcal B], d_{V_\mathcal B})$ so that $\theta_B\cdot f'$ and $f\cdot \theta_A$ are homotopic.
Define the value of the functor on $(\mathcal A, d_\mathcal A)$ to be the value on $(\Lambda[V_\mathcal A], d_{ V_\mathcal A})$ and the value on a morphism 
$f:(\mathcal A, d_\mathcal A)\to (\mathcal B, d_\mathcal B)$ to be the value on the morphism 
$f':(\Lambda[V_\mathcal A], d_{ V_\mathcal A})\to (\Lambda[V_\mathcal B], d_{V_\mathcal B}).$ 

There are two natural examples of mixed CDGA's; one is provided  by a smooth manifold equipped with a smooth vector field,  the other by a construction referred to as "the free loop",  considered first by Sullivan-Vigu\'e. The free loop construction applies directly only  to a fc-CDGA but in view of Theorem \ref{T1} can be indirectly used for any c--CDGA.
 
The first  will lead to  (the de--Rham version of)  a new homotopy functor defined on the category of possibly  infinite dimensional manifolds (hence on the homotopy category of all countable CW complexes) , the {\bf s-cohomology}, and  its relationship with other familiar 
homotopy functors \footnote {this functor was called in \cite{B2} and \cite{B4}  string cohomology for its unifying role  explained  below, cf. Observation \ref {O2}. The name "string homology" was afterwards used  by Sullivan and his school to designate the homology and equivariant homology of the free loop space of a closed manifold when endowed with  additional structures induced by intersection theory  and the Pontrjagin product in the chains of based pointed loops  cf. \cite {CS}.}, cf section \ref{SDeR} below.
 The second  leads  to simple definitions of three  homotopy functors  defined on the full category of c--CDGA's (via the minimal model theorem)  with values in the graded vector spaces endowed  with weight decomposition, cf  section \ref{Ssc} below.  Their properties  lead  to interesting results about cohomology of the free loop space of 1-connected spaces.
\vskip .2in

\section {De-Rham Theory  in the presence of a smooth vector field}\label{SDeR}

Let $M$ be a smooth manifold, possibly of infinite dimension.  In the last  case the manifold  is modeled on a good Frechet  space \footnote {this is a Frechet space with countable base which admits a smooth partition of unity; Note that if a Frechet space $V$ is good then the space of smooth maps $C^{\infty} (S^1, V)$ equipped with the $C^\infty-$ topology is also good.}, for which  the differential calculus can be performed as expected.  

Consider the CDGA  of differential forms $\Omega^\ast (M) $ with exterior differential $d^\ast: \Omega^\ast  (M) \to \Omega^{\ast +1}(M)$ and interior differential
$ i^X_\ast:\Omega^\ast  (M) \to \Omega^{\ast -1}(M),$ the contraction along the vector field $X.$
They are not compatible. However we can consider the Lie derivative $L_X:= d\cdot i^X + i^X\cdot d $ and define $\Omega_X(M):= \{\omega\in \Omega (M) | L_X\omega=0\};$ 
$\Omega_X(M)$ consists of the smooth forms invariant by the flow induced by $X.$ 
The graded vector space $\Omega^\ast_X (M)$ is a commutative graded algebra, a sub algebra of $\Omega ^\ast (M),$ and the restriction of $d^\ast $ and of $i^X_\ast$  leave invariant $\Omega^\ast_X(M)$ and are compatible. Consequently  $(\Omega^\ast_X(M), d^\ast,i^X_\ast)$ is a mixed CDGA.

Denote  by 
\begin{enumerate}
\item $H^\ast_X(M):= H^\ast (\Omega^\ast_X, d^\ast),$
\item $^\pm H^\ast_X(M):= ^\pm H^\ast (\Omega^\ast_X, d^\ast, i^X_\ast),$
\item $PH^\ast_X(M):= PH^\ast (\Omega^\ast_X, d^\ast, i^X_\ast),$
\item $\mathbb PH^\ast_X(M):= \mathbb PH^\ast (\Omega^\ast_X, d^\ast, i^X_\ast).$
\end{enumerate}

The diagram Fig 2 becomes 

\diagram
 \rto &^+H^r_X(M)\rto^{S^r}\dto^{id} & ^+H^{r+2}_X(M) \rto^{J^{r+2}} \dto^{\mathbb I^{r+2}} & H^{r+2}(M)\rto^{B^{r+2}}\dto                 & ^+H_X^{r+1} (M)  \rto \dto  &  \cdots\\
 \rto &^+H^r_X(M)\rto ^{\mathbb I^r}     & \mathbb PH^{r+2}(M)\rto^{\mathbb J^{r+2}}                          & ^-H^{r+2}_X(M)\rto^{\mathbb B^{r+2}}                      & ^+H^{r+1}_X(M)  \rto  &\cdots 
\enddiagram
\hskip 2in  Fig 3
\vskip .1in

The above diagram becomes  more interesting if the vector field $X$ is  induced from an $S^1$ action 
$\mu:S^1\times M\to M$ (i.e. if $x \in M$ then $X(x)$ is the  tangent to the orbits through $x$). 
In this section we will explore particular cases of this diagram. 

Observe that since $\mu$ is a smooth action,  the subset  $F$ of fixed points is a smooth sub manifold.  For any $x\in F$
denote by $\rho_x: S^1\times T_x(M)\to T_x(M)$  the linearization of  the action at $x$  which is a  linear representation. The inclusion $ F\subset M$ induces the morphism $r^\ast: (\Omega^\ast_X(M), d^\ast, i^X_\ast)\to 
(\Omega^\ast(F), d^\ast, 0).$ 

For  a linear representation $\rho: S^1\times V\to V$ on a good Frechet  space denote by $V^f$ the fixed point set and by $X$ the vector field associated to $\rho$ when regarded as a smooth action. 

\begin{definition}\label{D4}
A linear representation $\rho :S^1\times V\to V$ on the good Frechet space is good if 
the following conditions hold:

a. $V^f,$ the fixed point set, is a good Frechet space, 

b. The map $ r^\ast: \Omega^\ast (V) \to \Omega^\ast (V^f)$  induced by the inclusion is surjective, 

c.  $(\Omega^\ast_X (V, V^f), i^X_\ast)$ with $(\Omega^\ast_X (V, V^f) = \ker  r^\ast$  is acyclic.
\end{definition}


We have:
 
 \begin{proposition}\label{P2}
 
 1. Any representation on a finite dimensional  vector space is good.
 
 2. If $V$ is  a good Frechet  space then the regular representation $\rho: S^1 \times C^\infty (S^1, V)\to 
C^\infty (S^1, V)$, with $C^\infty (S^1, V),$ the Frechet space of smooth functions, is good.
\end{proposition}
For a proof consult Appendix  \cite{B1}. The proof is based on an explicit formula for  $i^X$ in the case of irreducible $S^1-$ representation and on the writing of the elements of $C^\infty (S^1, V)$ as Fourier series.

\begin{definition}\label{D5}
A  smooth action $\mu: S^1\times M\to M$  is good if its linearization at any fixed point is a good representation. 
\end{definition} 
 
Then a smooth action on any finite dimensional manifold is good and so is the canonical smooth action of $S^1$ on $P^{S^1},$ the smooth manifold of smooth maps from $S^1$ to $P$ where $P$ is any smooth Frechet  manifold  (in particular a finite dimensional manifold).
In view of the definitions above 
observe the following.
\begin {proposition}\label {P3}
If $\tilde M= (M,\mu)$ is a smooth $S^1-$manifold and $X$ is the associated vector field,  then:

1.  $H^\ast_X(M)= H^\ast (M),$

2. $^+ H^\ast_X(M)= H^\ast_{S^1}(\tilde M)$ and $ S:H^\ast(M)\to H^{\ast +2}(M)$ identifies to the multiplication with $u\in H^2_{S^1}(pt),$
the generator of the equivariant cohomology of one point space,

3. $PH^\ast_X(M)= \underset {\underset {S} {\rightarrow}} {\lim} H^{\ast+2k}_{S^1}(\tilde M).$
 
If the action is good then: 

4. $\mathbb PH^\ast_X(M)= K^\ast (F)\ $  
where 
\begin{equation*}
\begin{aligned}
K^r(F)= &\prod_k H^{2k}(F)\ \text {if}\ \ r \ \ even\\
K^r(F)= &\prod_k H^{2k+1}(F)\ \text {if}\ \  r \ \ odd.
\end{aligned}
\end{equation*}

If $M$ is a closed of $n-$dimension  manifold then :
 
5. $^- H^\ast_X(M)= H^{S^1}_{n-1-\ast} (\tilde M, \mathcal O_M)$
with $H^{S^1}_\ast (\tilde M, \mathcal O_M)$  the equivariant homology with coefficients in the orientation bundle  \footnote{ Recall that $H^{S^1}_{\ast}(M,\mathcal O_M)= H_\ast ( M// S^1, \mathbb O_M)$ where $M//S^1$ is the homotopy quotient of this action.  This equivariant homology can be derived from invariant currents in the same way as equivariant cohomology from invariant forms,  cf. \cite{AB}.
The complex of invariant currents (with coefficients in orientation bundle) contains the complex $(\Omega^{n-\ast} _X(M, \mathcal O_M), \partial _{n-\ast})$ as a quasi isomorphic sub complex.} of $M.$
\end{proposition}

\proof 
1. The verification is  standard since $S^1 $ is compact and connected; one construct  $av^\ast: (\Omega^\ast(M), d^\ast) \to (\Omega_X^\ast, d^\ast)$ by $S^1-$ averaging using the compacity of $S^1.$ The homomorphism induced in cohomology by $av^\ast$ is obviously surjective. To check  it is injective one has to show that any closed $k$ differential form $\omega$ which becomes  exact after applying $av$ is  already exact, precisely $\int_c\omega=0$ for any smooth $k-$ cycle $c.$ Indeed, since  the connectivity of $S^1$ implies $\int_c\omega= \int_{\mu(-\theta, c)}\omega , \ \theta \in S^1$
one has: 

\begin{equation*}
\begin{aligned}
\int_c\omega= 1/2\pi \int_{S^1}(\int_c\omega)d\theta=  \\
1/2\pi \int_{S^1}(\int_{\mu(-\theta, c)} \omega)d\theta
=1/2\pi  \int_{S^1}(\int_c \mu_\theta^\ast (\omega))d\theta= \\
 \int_c (1/2\pi \int_{S^1} \mu_\theta^\ast (\omega)d\theta)= 
\int_c av^\ast(\omega)=0.
\end{aligned}
\end{equation*} 
Here $\mu_\theta$ denotes the diffeomorphism $\mu(\theta, \cdot ): M\to M.$ 

2. Looking at the definition in section \ref{SMC} one  recognizes one of the most familiar definition of equivariant cohomology using invariant differential forms cf. \cite{AB}. 

3. The proof is a straightforward consequence of Proposition \ref{P1} in section \ref{SMC} and (2) above.

4. Let $F$ be the smooth sub manifold of the fixed points of $\mu.$  Clearly $r^\ast : (\Omega^\ast_X(M), d^\ast, i^X_\ast) \to (\Omega^\ast (F), d^\ast_F, 0)$ is a morphism of mixed CDGA, hence of mixed complexes.
If the smooth action is good then 
the above morphism induces an isomorphism in homology $H_\ast (\Omega^\ast_X, i^X_\ast)\to H_\ast (\Omega^\ast (F), 0).$  To check this  
we have to show that $(\ker r^\ast, i^X_\ast)$, with $\ker r^\ast:= \{\omega\in \Omega_X ^\ast(M) | \omega|_F=0\},$ is acyclic.  This follows (by $S^1-$average) from the acyclicity of the chain complex $(\Omega^\ast (M,F), i^X_\ast)$ which in turn can be derived using  the linearity  w.r. to functions of of $i^X.$
Indeed, using  a "partition of unity"  argument it suffices to verify this acyclicity  locally.
For points outside $F$  the acyclicity follows from the acyclicity of the complex 

$\cdots \Lambda^{\ast-1} (V) \overset {i^e}\to \Lambda^\ast(V)\overset{i^e}\to \Lambda ^{\ast-1}(V)\to \cdots , $
where $V$ is a Frechet space , $\Lambda^k (F)$ the space of skew symmetric $k-$linear maps from $V$ to $\kappa= \mathbb R, \mathbb C$ and $e\in V\setminus 0.$ For points $x\in F$ this follows from the
fact that  the linearization of the action at  $x$ is a good representation, as stated in Proposition \ref{P2}.

5. If $\tilde M$ is a finite dimensional smooth $S^1-$manifold we can equip $M$  with an invariant Riemannian metric  $g$ and consider  $\star:\Omega^\ast(M)\to \Omega^{n-\ast}(M; \mathcal O_M)$ the Hodge star operator.
Denote by $\omega\in \Omega^1 (M)$ the $1-$form corresponding  to $X$ w.r. to  the metric $g$,   by 
$e_\omega:\Omega^\ast(M;\mathcal O_M)\to \Omega^{\ast +1}(M;\mathcal O_M)$ the exterior multiplication with $\omega$ and by $\partial_\ast: \Omega^\ast(M;\mathcal O_M)\to \Omega^{\ast-1}(M;\mathcal O_M)$ the formal adjoint of $d^{\ast-1}$ w.r. to $g$ i.e. 
$\partial _\ast  =  \pm \star \cdot\ d ^{n-\ast}\cdot \ \star^{-1} .$
Note that $e_\omega= \pm  \star  \cdot \   i^X \cdot \star^{-1} .$ 
All these operators leave $\Omega_{X} $ invariant since $g$ is invariant.
Clearly $(\Omega ^\ast_{X}(M;\mathcal O_M), e^\ast _\omega, \partial _\ast ) $ is a mixed cochain complex
and we have 
$$^-H^\ast _{i^X}(\Omega_X (M), d, i^X)= ^+H^\partial_{n-\ast} (\Omega_X (M;\mathcal O_M),  e^\omega,\partial ).$$
The equivariant   homology
of $\tilde M$  with coefficients in the orientation bundle can be calculated from the complex of invariant currents  which, if $M$ closed, contains the complex $(\Omega^{n-\ast} _X(M, \mathcal O_M), \partial _{n-\ast})$ as a quasi isomorphic sub complex.
As a consequence 
we have 
\begin{equation*}
\begin{aligned}
H_{n-\ast}(\Omega_X(M;\mathcal O_M),\partial)= &H_\ast(M;\mathcal O_M)\\
^+H^{e_\omega}_{n-\ast}(\Omega_X(M;\mathcal O_M), e_\omega,\partial)= &H_\ast^{S^1}(M;\mathcal O_M)
\end{aligned}
\end{equation*}
(cf. section \ref{SMC} for  notations).
\endproof

As a consequence of Proposition \ref{P3}  (1)-(4) for any smooth $S^1$ manifold with good $S^1-$ action the second long exact sequence  in diagram Fig 3  becomes

\diagram 
\cdots \to  H^{r-2}_{S^1}(\tilde M)\drto \rto^{\mathbb I^{r-2}} &K^{r}(F) \rto^{\mathbb J^r}  &{^-H}^{r}_{S^1}(\tilde M)\rto^{\mathbb B^r} &H^{r-1}_{S^1}(\tilde M)\rto &\cdots\\
                                                                   &\underset{\underset{S}{\rightarrow}}\lim H^{r+2k}_{S^1}(\tilde M)\uto 
\enddiagram
\hskip 2in Fig 4
\vskip .1in
The sequence above  is obviously natural in the sense that $f:\tilde M\to \tilde N,$ an $S^1-$ equivariant smooth map, induces a commutative diagram whose rows are the above exact sequence Fig 4 for $ \tilde M$ and $\tilde N.$
Then if $f$ and  its restriction to  the  fixed point set induce  isomorphisms  in cohomology   it induces  isomorphisms in $H^\ast_{S^1}$ and $K^\ast$ and then all other  types of equivariant cohomologies  $ ^- H^\ast_{S^1}, \ \ P H^\ast_{S^1}, \ \ \mathbb PH^\ast_{S^1}.$

If $\tilde M$ is a compact smooth $S^1-$ manifold in view of  Proposition \ref{P3}  (5)  one identifies 
$^-H^r_{S^1}(M)$ to $H^{S^1}_{n-r}(M;\mathcal O_M)$ and in view of this identification write $Pd_{n-r}.$
instead of $\mathbb B^r.$ The long exact sequence becomes  

\diagram 
\cdots \rto &K^r(F) \rto  &{H}_{n-r}^{S^1}(\tilde M; \mathcal O_M)\rto^{Pd_{n-r}} &H^{r-1}_{S^1}(\tilde M)\rto & K^{r+1}(F)\rto&\cdots .
\enddiagram
\hskip 2 in Fig 5
\vskip .1in

In case  that  the fixed point set is empty  we conclude that 
$$ Pd_{n-r}: {H}_{n-r}^{S^1}(\tilde M, \mathcal O_M) \to H^{r-1}_{S^1}(\tilde M)$$ is an isomorphism. In this case
the orbit space $M/ S^1$ is a $\mathbb Q-$homological manifold of dimension $(n-1)$, hence 
\begin{equation*}
\begin{aligned}
H_{n-r}^{S^1}(\tilde M;\mathcal O_M) =& H_{n-r}(M/ S^1 ;\mathcal O_{M/S^1})\\ 
H_{S^1}^{r-1}(\tilde M;\mathcal O_M) =& H ^{r-1}(M/ S^1 ;\mathcal O_{M/S^1})
\end{aligned}
\end{equation*}
 and  $Pd_{\ast}$ is  nothing but  the Poincar\'e   duality isomorphism for $\mathbb Q-$ homology manifolds.
In general the  long exact sequence Fig 5  measures the failure of the Poincar\'e duality map, $Pd_{\ast},$ to be an isomorphism.
\vskip .2in

\section{The free loop space and s-cohomology} \label{Ssc}

A more interesting example  is provided by the $S^-$manifold $\tilde {P^{S^1}}:= (P^{S^1}, \mu).$
Here $P^{S^1}$ denotes the  smooth manifold 
of smooth maps from $S^1$ to $P$  modeled by the  Frechet space $C^\infty(S^1, V)$ where $V$ is the  model for $P$ (finite or infinite dimensional Frechet space) cf \cite{B1}. This smooth manifold  is equipped with the canonical smooth $S^1-$action $\mu :S^1\times P^{S^1} \to  P^{S^1}$ defined by 
$$\mu(\theta, \alpha)(\theta')= \alpha (\theta+\theta'),\ \ \alpha: S^1\to P,\ \  \theta, \ \theta' \in S^1= \mathbb R/ 2\pi.$$    The  fixed points set of the action $\mu$ consists of the constant maps hence identifies with $P.$  This action is the restriction of the canonical action of $O(2),$ the group of isometries  of $S^1,$  to the subgroup of  orientation preserving isometries identified to  $S^1$ itself.  For any $x\in P$ viewed as a fixed point of $\mu$ the linearization representation is the regular representation of $S^1$ on $V= T_x(P).$  In view of Proposition \ref{P2} the action $\mu$ is good. 
The space $P^{S^1}$ is also equipped with the natural maps $\psi_k, \ k=1,2,\cdots,$ the geometric power maps and with the involution $\tau,$  
defined by

\begin{equation*}
\begin{aligned}
\psi_k (\alpha)(\theta)= \alpha(k\theta)\\
\tau (\alpha)(\theta)= \alpha(-\theta)
\end{aligned}
\end{equation*}
with  $\alpha\in P^{S^1}$, and $\theta\in S^1.$

The involution $\tau$ is the restriction of the action of $O(2)$ to the reflexion  $\theta\to -\theta$
in $S^1.$ 
Then  $(\Omega^\ast _X(P^{S^1}), d^\ast, i^X_\ast)$ is a mixed CDGA, hence a mixed complex with power maps  $\Psi_k, \tau $ and involution $\tau$ induced from $\psi_k$ and $\tau.$ 

Suppose $f:P_1\to P_2$ is a smooth map. It  induces a smooth equivariant map  
$f^{S^1} : {P_1^{S^1} }\to {P_2^{S^1}}$  whose restriction to the fixed points set is exactly $f.$  
If $f$ is a homotopy equivalence  then so is $f^{S^1}.$ 

Introduce the notation 
\begin{equation*}
\begin{aligned}
hH^\ast(P):= H^\ast (P^{S^1}),\\
cH^\ast(P):= H^\ast_{S^1} (\tilde{P^{S^1}}),\\
sH^\ast (P):= ^- H^\ast _{S^1}(\tilde {P^{S^1}}).
\end{aligned}
\end{equation*}
The assignments  $P\rightsquigarrow  hH^\ast(P),$ $P\rightsquigarrow  cH^\ast(P),$          
$P\rightsquigarrow  sH^\ast(P)$ are functors 
\footnote {the notations $hH^\ast, cH^\ast$ are motivated by the Hochschild resp. cyclic homology interpretation of these functors,
while $sH^\ast$ is abbreviation from string cohomology.}  with the property that $hH^\ast (f), cH^\ast (f), sH^\ast(f)$ are   isomorphism if $f$ is a homotopy equivalence, hence they are all homotopy functors . They are related by the commutative diagram below. This diagram  is the same as diagram (Fig 3) applied to $\tilde M= \tilde{P^{S^1}}$ with
the specifications provided by Proposition \ref{P3}.

\diagram
\cdots \rto&cH^{r-2}(P)\rto^{S^{r-2}}\dto^{id} & cH^{r}(P) \rto ^{J^r} \dto^{\mathbb I^r} & hH^{r}(P)\rto^{B^r}\dto   & cH^{r-1}(P)   \rto \dto^{id}  &  \cdots\\
\cdots \rto&cH^{r-2}(P)\rto ^{\mathbb I^{r-2}}\dto ^S        & K^r (P)\rto^{\mathbb J^r}                            & sH^{r}(P)\rto^{\mathbb B^r}                        & cH^{r-1}(P)  \rto  &\cdots \\
&cH^r(P)\rto \urto^{\mathbb I^r}&\underset{\rightarrow} \lim\  cH^{r+2k}(P)\uto
\enddiagram 
\hskip 2in Fig 6

\noindent where
$\underset{\rightarrow} \lim \ \ cH^{r+2k}(P)= \underset{\rightarrow} \lim \{\cdots \to cH^{r+2k}(P)\overset {S}{\rightarrow} cH^{r +2k+2}(P)\to \cdots\}.$ 
 
The linear map $$cH^\ast(P)= H^\ast_{S^1}(\tilde {P^{S^1})}\overset {\mathbb I^r}{\rightarrow} K^\ast$$   factors through $\underset{\rightarrow}{\lim}\  cH^{r+2k}(P)$ which depends  only on the fundamental group of $P.$  Indeed, it is shown in \cite {B3} that  if $P(1)$ \footnote {the notation for the first stage Postnikov term of $P.$} is a smooth manifold (possibly of infinite dimension) which has the homotopy type of $K(\pi,1)$ and $p(1): P\to P(1)$ is smooth map inducing
an isomorphism for the fundamental group then 
$\underset{\rightarrow}{\lim} H^{r+2k}_{S^1}(\tilde {P^{S^1}}) \to\underset{\rightarrow}{\lim}H^{r+2k}_{S^1}(\tilde {P(1)^{S^1}}) $ is an isomorphism. 
cf \cite {B3}.
Then if one denotes by $\overline {cH^\ast}(M):= \text{coker} (cH^\ast (M)\to cH^\ast (pt))$ and  
$\overline {K^\ast}(M):= \text{coker} (K^\ast (M)\to K^\ast (pt))$\footnote { clearly $K^r (pt)= H^r_{S^1}(pt)= 
\kappa \ \ \text{resp.} \ \ 0$ if $r$ is even resp. odd} one  obtains

\begin {theorem}\label{T2}
If $P$ is a 1-connected smooth manifold then we have the following short exact sequence:
 $$ 0 \to   \overline K^r(P)\otimes \kappa \overset {\mathbb J^r}\rightarrow  {sH}^r (P) \overset {\mathbb B^r}\rightarrow  \overline cH^{r-1}(P)\to 0$$
where $\kappa= \mathbb R$ or $\mathbb C.$
\end{theorem} 

 \begin{observation} \label{O2}
 The vector space $K^\ast(P)$ can be identified via the Chern character to the Atiyah--Hirtzebruch (complex) $K-$theory tensored with the field $\kappa= \mathbb R \text{or} \ \mathbb C,$ depending of what sort of differential forms one consider (real or complex valued).  When $P$ is 1-connected $\overline cH^\ast (P)$ identifies to $\text{Hom}(\tilde A_{\ast}(P), k)$ where $\tilde A_{\ast}(P)$
denotes the reduced Waldhaussen algebraic  $K-$theory\footnote {often referred to as $A-$ theory.}, cf \cite{B2}. From this perspective $sH^\ast$ unifies topological (Atiyah--Hirtzebruch) $K-$theory  and Waldhaussen algebraic $K-$theory.
\end{observation}

\begin{observation}\label {O3}
In view of the definition of $^-H^\ast_\beta(C^\ast, \delta^\ast, \beta_\ast),$  cf. section \ref {SMC},  observe that 
$sH^\ast(P)$ is represented by infinite sequences 
\footnote {$sH^\ast (P)$ is the cohomology of the cochain complex $(^-C^\ast, ^-D^\ast)$
with $^- C^r= \prod _{k\geq 0} \Omega^{r+2k}_{inv} (P^{S^1})$ and 
$^-D^r(\cdots,   \omega_{r+2}, \omega_{r})= (\cdots,   (i^X \omega_{r+2} + d \omega_r)),$ cf. section \ref{SMC}.},
rather than eventually finite sequences  of invariant differential forms on  $P^{S^1}.$  If instead of ``infinite sequences"  we would have considered  ``eventually finite sequences''  the outcome would have been different for  infinite dimensional manifolds.
The difference between ``infinite sequences"  and "eventually finite sequences"  exists only for infinite dimensional manifolds which $P^{S^1}$ always is.
\end{observation} 


\vskip .2in
The power maps $\psi_k$ induce the endomorphisms  $h\Psi_k,$ $c\Psi_k$ $s\Psi_k$
and $K\Psi_k$ on  $hH^\ast, cH^\ast, sH^\ast,$ and $K^\ast.$

In general only  $K\Psi_k$ 
are easy to describe. Precisely 
if $r$ is even then  $K^r=  \prod_{i\geq 0} H^{2i}(P)$ 
and if $r$ is odd  then $K^r=  \prod_{i\geq 0} H^{2i+1}(P),$ and in both cases 
$K\Psi_k= \prod_{i\geq 0} k^{i-r}Id.$

The symmetric 
part   with respect to the involution $c\Psi_{-1},$  i.e. the eigenspaces corresponding to the 
eigenvalues $+1$ 
identifies with 
$H^\ast _{O(2)}(P^{S^1}),$ 
the  equivariant cohomology for the canonical $O(2)-$action.

However, if $P$ is 1-connected, in view of the  section \ref {SFree},  one can  describe both the eigenvalues and the eigenspaces of the power maps $h\Psi_k$ and $c\Psi_k$ and then of $s\Psi_k.$
We have: 
\begin{theorem}\label{T3}

Let $P$ be a 1-connected manifold.

1. All  eigenvalues of the endomorphisms $ h\Psi_k$ and $ c\Psi_k$ are $k^r, r=0,1, 2\cdots,$ and the eigenspaces  corresponding to $k^r$ are independent of $k$ provided $k\geq 2.$ 

2. Denotes these eigenspaces  by $hH^\ast(M)(r)$ and $ cH^\ast (M)(r).$ Then 

$hH^\ast(0)= H^\ast(X;\kappa),  cH^{\ast}(0)= H^{\ast+1}(X;\kappa),$ and 

$hH^r(p)=  cH^r(p)= 0, p\geq r+1.$

3. If $\sum_i \dim \pi_i(P)\otimes \kappa < \infty, $ $\kappa$ the field of real or complex numbers and 
$\sum_i \dim (H^i(P) )<\infty$ then for any $r\geq 0$ one has 
$$\sum_i \dim hH^i(P) (r) < \infty, \ \ \sum_i\dim cH^i(P)(r) < \infty.$$
\end{theorem}

If $P$ is "formal" in the sense of rational homotopy theory \footnote{  i.e.  for each connected component  of $P$ De-Rham algebra and the cohomology algebra equipped with  the differential $0$ are homotopy equivalent, cf section \ref{SFree}.}, (a projective complex algebraic variety,
or more general a closed Kaehler manifold  is formal, cf \cite {DGMS} ) then the Euler Poincar\'e characteristic 
$$\chi^h(\lambda) := \sum _{i, r} \dim hH^i(r)\lambda^r$$ and $$\chi^c(\lambda):= \sum _{i, r} \dim cH^i(r)\lambda^r$$ can be explicitly calculated in terms of the numbers $\dim H^i(P),$ cf [B].  The explicit formulae  are  quite complicated. They require the results of P.Hanlon \cite {H} about the eigenspaces of Adams operations in Hochschild and cyclic homology as well as  the identification of $hH^\ast (P)$ resp. $cH^\ast (P)$ with the Hochschild resp. cyclic homology of the graded algebra $H^\ast(P).$ These are  not discussed in this paper but the reader can consult  \cite {BFG} 
and  \cite {B4} for precise statements.

The functor $\overline {sH}^r (P)$ is of particular  interest  in geometric topology. In  the case $P$ is 1-connected
it calculates in some ranges the homotopy groups of the  (homotopy) quotient space of homotopy equivalences  by the group of diffeomorphisms   \cite{B1}, \cite {B4}.

 \section {The free loop construction for CDGA} \label{SFree}

The "free loop " construction associates  to a free connected CDGA, \  $(\Lambda[V], d_V)$
 a  mixed CDGA,  $(\Lambda[V\oplus \overline V],\delta_V, i^V),$  endowed with   power maps $\Psi_k$ and involution $\tau$ defined as follows.

\begin{enumerate}
\item Let $\overline V=\oplus _{i\geq 0}\overline V^i$ with $\overline V^{i}:= V^{i+1}$
and let $\Lambda[V\oplus \overline V]$ be  the commutative graded algebra generated by $V\oplus \overline V.$ 
\item Let $i^V: \Lambda[V\oplus \overline V] \to \Lambda[V\oplus \overline V]$be  the unique internal differential (of degree $-1$) which extends $i^V(v)= \overline v$ and $i^V(\overline v)=0.$
\item 
Let $\delta_V: \Lambda[V\oplus \overline V] \to \Lambda[V\oplus \overline V]$ be the unique external differential (of degree $+1$)  which extends $\delta_V(v)=d (v)$ and $\delta (\overline v)=  - i^V(d(v)).$
\item Let $\Psi_k: (\Lambda[V\oplus \overline V], \delta_V)\to 
(\Lambda[V\oplus \overline V], \delta_V), k=-1, 1,2,\cdots$ be the unique morphisms of CDGA 
which extends $\Psi_k(v)= v, \Psi_k(\overline v)= k\overline v.$ We put $\tau:=\Psi_{-1}.$ The maps 
$\Psi_k$ $k\geq 1$ are called the power maps and $\tau$ the canonical involution. 
One has 
\begin{equation*}
\begin{aligned}
\Psi_k\cdot \Psi_r=& \Psi_{kr}\\
\Psi_k \cdot i_V=& k i_V\cdot \Psi_k
\end{aligned}
\end{equation*}
\item Let $\Lambda^+[V\oplus \overline V]$ be the ideal of $\Lambda[V\oplus \overline V]$  generated by 
$V \oplus \overline V$ or the kernel of the  augmentation which vanishes on $V \oplus \overline V.$
\end{enumerate}

Note that :
\begin{observation}\label{O3}

1.  $\text({Im}(i^V) ,\delta_V, 0)$ is a mixed sub complex of ($\Lambda^+[V \oplus \overline V],\delta_V, i^V)\subset (\Lambda[V\oplus \overline V],\delta_V, i^V)$

2. $\Psi_k$, $ k= -1, 1, 2,\cdots$  leave $(\Lambda^+[V\oplus \overline V],\delta_V, i^V) $ and $(\text{Im}(i^V) ,\delta)$  invariant  and have $k^r$ $r=0,1,2,\cdots$ as  eigenvalues.  These  are all eigenvalues. 

For $k \geq 2$ the eigenspace of $\Psi_r: \Lambda[V\oplus \overline V]\to \Lambda [V\oplus \overline V]$ corresponding to the eigenvalue $k^r$ is exactly 
$\Lambda[V]\otimes \overline V^{\otimes r},$ resp. $\Lambda^+[V\oplus \overline V]\cap \Lambda^[V]\otimes \overline V^{\otimes r},$
resp. $\text{Im}(i^V)(r)= \text{Im}(i^V)\cap \Lambda [V] \otimes \overline V^{\otimes r}$,  hence independent of $k.$ Each such eigenspace is $\delta_V-$invariant.

3. The mixed complex $(\Lambda^+[V\oplus \overline V],\delta_V, i^V)$ is $i^V-$acyclic,

4. We have the decomposition 

$$(\Lambda [V\oplus \overline V], \delta_V)= \bigoplus_{r\geq 1} (\Lambda[V]\otimes \overline V^{\otimes r}, \delta_V)$$
and the analogous decomposition for $(\Lambda^+ [V\oplus \overline V], \delta_V)$
and $(\text{Im}(i^V), \delta_V)$ referred from now on as the weight decompositions. 
\end{observation}
Consider the complex $(\Lambda[V]\otimes\overline V^{\otimes r}, \delta_V)$ and the filtration provided by $\Lambda[V]\otimes F_p(\overline V^{\otimes r})$ with $ F_p(\overline V^{\otimes r})$ the span of elements in $\overline V^{\otimes r}$ of total degree $ \leq p.$ 

For a graded vector space $W=\oplus _i W^i$
denote by $\dim W= \sum \dim W^i.$

\begin{observation}\label {O4}

1. $(\Lambda[V]\otimes F_p(\overline V^{\otimes r}), \delta_V)$ is  a sub complex of $(\Lambda[V]\otimes \overline V^{\otimes r}, \delta_V).$ 

2.If $(\Lambda[V], d_V)$ is minimal and one  connected then, by Observation \ref{O0}, 

$\delta (F_p(\overline V^{\otimes r})) \subset \Lambda[V]\otimes F_{p-1}(\overline V^{\otimes r})$
and then  
$$(  {\Lambda[V]\otimes F_p(\overline V^{\otimes r})}/{  \Lambda[V]\otimes F_{p-1}(\overline V^{\otimes r})} \ , \delta_V)= (\Lambda[V], d_V)\otimes {F_p(\overline V^{\otimes r})}/ {F_{p-1}(\overline V^{\otimes r})}.$$

2. $\sum_p  \dim (F_p(\overline V^{\otimes r}) / F_{p-1}(\overline V^{\otimes r}))= 
\dim (\overline V^{\otimes r})=
(\dim V)^r.$ 
\end{observation}

If $f: (\Lambda[V], d_V)\to \Lambda(W, d_W)$ is a morphism of CDGAs then it induces 
$\tilde f: (\Lambda[V\oplus \overline V], \delta_V, i^V)\to( \Lambda[W \oplus \overline W], \delta_W, i^W)$
which intertwines $\Psi_k'$s and then preserves the weight decompositions.

We introduce the the notation 
$HH^\ast, CH^\ast PH^\ast $ 

\begin{equation*}
\begin{aligned}
HH^\ast (\Lambda[V], d_V):=  & H^\ast (\Lambda[V\oplus \overline V], \delta_V)\\
CH^\ast (\Lambda[V], d_V):=  & ^+H^\ast_{i^V} (\Lambda[V\oplus \overline V], \delta_V, i^V)\\
PH^\ast (\Lambda[V], d_V):=  & PH^\ast_{i^V} (\Lambda[V\oplus \overline V], \delta_V, i^V)\\
\end{aligned}
\end{equation*} 
and for a morphism $f$  denote by  $HH(f), CH(f), PH(f)$ the linear maps induced by $\tilde f.$  
The assignments $HH^\ast, CH^\ast, PH^\ast$ 
provide functors from the category of fc--CDGA's
to graded vector spaces. 
They come equipped with the 
operations $H\Psi_k,  C\Psi_k $ etc. induced from $\Psi_k.$ Since for $f$ quasi isomorphisms 
$HH^\ast (f), CH^\ast (f),$ $ PH^\ast(f)$ are isomorphisms  these functors,
 as shown  in section \ref{SMA}  extend to the  category of  c--CDGA's.  
 We have the following result. 

\begin{theorem} \label{T4}
Let $(\mathcal A, d_\mathcal A)$ be a connected CDGA. 

1. All  eigenvalues of the endomorphisms $ H\Psi_k$ and $ C\Psi_k$ are $k^r, r=0,1, 2\cdots,$ and their eigenspaces  are independent of $k$  provided  $k\geq 2.$
One denotes them by $HH(\mathcal A, d_\mathcal A)(r),$ and  $ CH(\mathcal A, d_\mathcal A)(r).$

2. \begin{equation*}
\begin{aligned}
HH^\ast(\mathcal A, d_\mathcal A)(0)=&H^\ast(\mathcal A, d_\mathcal A), \\
  CH^\ast(\mathcal A, d_\mathcal A)(0)= &H^{\ast+1}(\mathcal A, d_\mathcal A)\\
HH^r(\mathcal A, d_\mathcal A)(p)= &CH^r(\mathcal A, d_\mathcal A)(p)=0, \ \   p\geq r+1
\end{aligned}
\end{equation*}

3.Suppose  $(\mathcal A, d_\mathcal A)$ is 1-connected with minimal model $(\Lambda[V], d_V).$
If  $ \sum_i\dim V^i< \infty$ and $\sum_i \dim H^i(\mathcal A, d_\mathcal A) <\infty$ then for any $r\geq 0$ one has 
$$\sum_i \dim HH^i(\mathcal A, d_\mathcal A)(r) < \infty, \ \ \sum_i\dim CH^i(\mathcal A, d_\mathcal A)(r) < \infty.$$
\end{theorem}

\proof

It suffices to check the statements for $(\mathcal A, d_\mathcal A)=(\Lambda[V],d_V)$ minimal.
Items 1) and 2)  are immediate consequences of Observation  \ref {O3}.

Item 3) follows from Observation \ref {O4}.  Indeed  for a fixed $r$ one has

\begin{equation*}
\begin{aligned}
\sum_{i} \dim H^i (\Lambda[V]\otimes \overline V^{\otimes r}, \delta_V)) \leq \\
\sum_{i, p} \dim H^i ({\Lambda[V]\otimes F_p(\overline V^{\otimes r}}/{ \Lambda[V]\otimes F_{p-1}(\overline V^{\otimes r})}, \  \delta_V)= \\
(\dim V)^r \cdot \sum_{i} \dim H^i  (\Lambda[V], d_V)  
\end{aligned}
\end{equation*}
\endproof


\noindent In addition to $\chi (\mathcal A, d_\mathcal A): =\sum (-1)^i\dim H^i(\mathcal A, d_\mathcal A)$
one can consider 

$\chi^H(\mathcal A, d_\mathcal A)(r) := \sum (-1)^i\dim HH^i(\mathcal A, d \mathcal A)(r)$ and 

$\chi ^C(\mathcal A, d_\mathcal A)(r): =\sum (-1)^i\dim CH^i(\mathcal A, d_\mathcal A)(r),$

\noindent and then the power series in $\lambda$, 
$$\chi^H (\mathcal A, d_\mathcal A)(\lambda):= \sum  \chi^H(\mathcal A, d_\mathcal A)(r) \lambda^r, \ \   \chi^C(\mathcal A, d \mathcal A)(\lambda) := \sum  \chi^C(\mathcal A, d \mathcal A)(r) \lambda^r.$$
Theorem \ref{T4}  (3) implies that for $(\mathcal A, d_\mathcal A)$ 1-connected  with $ \sum_i \dim V^i< \infty$ and $\sum_i \dim H^i(\mathcal A, d_\mathcal A) <\infty$ the partial Euler--Poincar\'e characteristics  $\chi^H (\mathcal A, d_\mathcal A)(r)$ and $\chi^C (\mathcal A, d_\mathcal A)(r)$ and therefore 
the  power series $\chi^H (\mathcal A, d_\mathcal A)(\lambda)$ and $\chi^C (\mathcal A, d_\mathcal A)(\lambda)$ are well defined.
The results of Hanlon \cite{H} permit to calculate explicitly $\chi^H(\lambda)$ and $\chi^C(\lambda)$ in terms of $\dim H^i(\mathcal A,d_A)$ if  $(\mathcal A, d_\mathcal A)$ is 1-connected and formal,  i.e. there exists a quasi isomorphims $(\Lambda[V],d)\to (H^\ast (\Lambda[V], d), 0),$ $(\Lambda[V],d)$ a minimal model of $(\mathcal A, d_\mathcal A).$ 
\vskip .1in

We want to define an algebraic analogue of the functor $sH^\ast$  on the category of cCDGA's. 
Recall that for a morphism $f^\ast:(C^\ast_1, d^\ast_1) \to (C^\ast_2, d^\ast_2)$ the ``mapping cone"
$\textit {Cone}(f^\ast)$ is the cochain complex with components $C^\ast_f= C^\ast_2 \oplus C^{\ast+1}_1$  and with

$d^\ast_f= \begin{pmatrix} d^\ast_2 & f^{\ast+1}\\
0&-d^{\ast+1}_1
\end{pmatrix}  .$  

Notice that, when $f^\ast$ is injective, the morphism $\textit{Cone}(f^\ast) \to C^\ast_2/ f^\ast(C^\ast_1)$  defined by the composition  $C^\ast_2 \oplus C^{\ast+1}_1\rightarrow C^\ast_2\rightarrow C^\ast_2/ f^\ast (C^\ast_1)$ is a quasi isomorphism.

We will consider  the composition
$$\underline I^{\ast-2} : ^+ \mathcal C^{\ast-2}_{i^V} (\Lambda[V\oplus \overline V],\delta_V, i^V)\overset{I^{\ast-2}} \rightarrow\mathbb P\mathcal C^\ast(\Lambda[V\oplus\overline V],\delta_V, i^V)\overset{\mathbb P^\ast(p)}{\rightarrow}  \mathbb P\mathcal C^r(\Lambda[V], d_V , 0)$$
with the fist arrow provided by the natural transformation $I^{\ast-2}: ^+C^{\ast-2}_\beta \to \mathbb P C^\ast$  described in section \ref {SMC} applied to  the mixed complex $(\Lambda[V\oplus\overline V],\delta_V, i^V)$ and the second induced by the projection on the zero weight  component of $(\Lambda[V\oplus\overline V],\delta_V, i^V).$

The mapping cone 
$\textit{Cone}(\underline I^{\ast-2}),$  is  functorial when regarded on the category of  fc--CDGA's.
Define 
$$SH^\ast (\Lambda[V], d_V):= H^\ast (\textit{Cone} (\underline I^\ast)) .$$ The assignment $(\Lambda[V], d_V)\rightsquigarrow SH^\ast  (\Lambda[V], d_V)$
is  an  homotopy functor. 

 Consider the  commutative diagrams 

\diagram 
 & ^+ \mathcal C^{\ast-2}_{i^V}(\Lambda[V\oplus \overline V],\delta_V, i^V)  \rto^{i^{\ast-2}}\dto^{id} & ^+\mathcal C^\ast_{i^V}\rto\dto^{\underline I^\ast}(\Lambda[V\oplus \overline V],\delta_V, i^V)\rto &\textit{Cone}(i^{\ast-2}) \dto \\
 &^+ \mathcal C^{\ast -2}_{i^V}(\Lambda[V\oplus \overline V],\delta_V, i^V) \rto^{\underline I^{\ast-2}} &\mathbb P \mathcal C^\ast \rto(\Lambda[V],d_V, 0) & \textit{Cone}(\underline I ^{\ast-2})   .
\enddiagram

\noindent and 

\diagram 
 &^+  \mathcal C^{\ast-2}_{i^V}(\Lambda[V\oplus \overline V],\delta_V, i^V)\rto^{i^{\ast-2}}\dto^{id} &^+ \mathcal C^\ast_{i^V}(\Lambda[V\oplus \overline V],\delta, i^V)\rto\dto^{id} & \textit{Cone}(i^{\ast-2})\dto\\
 &^+  \mathcal C^{\ast-2}_{i^V}(\Lambda[V\oplus \overline V],\delta_V, i^V)\rto^{i^{\ast-2}} &^+ \mathcal C^\ast_{i^V}(\Lambda[V\oplus \overline V],\delta_V, i^V)\rto &C^\ast(\Lambda[V\oplus\overline V], \delta_V)
\enddiagram

\noindent with the last vertical arrow  in the second diagram a quasi
isomorphism   as noticed above. 

The long exact sequence induced by passing to cohomology in the first diagram combined with the identifications implied by the second diagram lead to 

\diagram
\cdots \rto&CH^{r-2}\rto^{S^{r-2}}\dto^{id} & CH^{r} \rto ^{J^r} \dto^{T^r} & HH^{r}\rto^{B^r}\dto                 & CH^{r-1}   \rto \dto^{id}  &  \cdots\\
\cdots \rto&CH^{r-2}\rto ^{\mathbb I^{r-2}}\dto ^{S^{r-2}}        & K^r\rto^{\mathbb J^r}                            & SH^{r+2}\rto^{\mathbb B^r}                        & CH^{r-1}  \rto  &\cdots \\
&CH^r\rto \urto^{\mathbb I^r}&PH^r\uto
\enddiagram 
\hskip 2in Fig 7
\vskip .1in 
\noindent with $PH^r= \underset {\rightarrow}{\lim} CH^{r+2k}$
and $K^r:= K^r(\Lambda[V], d_V)$ given by  $=   \prod_k H^{2k} (\Lambda[V], d_V)$ resp.
 $ \prod_k H^{2k+1} (\Lambda[V], d_V)$if $r$ is even resp. odd.
 It is immediate that Theorem \ref{T2} remains true for $sH^\ast, cH^\ast$ replaced by $SH^\ast, CH^\ast$
 as follows easily from diagram Fig 7. 
The  diagram Fig 7 should be compared with diagram Fig 6. This explains why $SH^\ast (\Lambda[V],d^V)$ will be regarded as the algebraic analogue of $sH^\ast (P).$

It is natural to ask if the functors $HH^\ast, CH^\ast, SH^\ast$  applied to $(\Omega^\ast(P), d^\ast)$ calculate $hH^\ast, cH^\ast, sH^\ast$ applied to $P$ and the diagram Fig 7 identifies to the diagram Fig 6.
The answer  is   in general  no,  but is yes if  $P$ 1-connected .
 
The minimal model theory,  discussed in the next section, 
 permits to identify,  $HH^\ast(\Omega^\ast(P), d^\ast), CH^\ast(\Omega^\ast(P), d^\ast)$ to 
$hH^\ast(P), CH^\ast(P)$ and then $SH^\ast(\Omega^\ast(P), d^\ast)$ to $ SH^\ast(P)$
and actually diagram Fig 6 to diagram Fig 7 when $P$ is 1-connected.

\vskip .2in 

\section{Minimal models and the proof of Theorem \ref {T3}} \label{S:minimal}

Observe that if $(\mathcal A^\ast, d^\ast, \beta_\ast)$ is a mixed CDGA equipped with the power maps 
and involution $\Psi_k, k=-1, 1, 2, \cdots,$ then the diagram

\diagram
&^+H^\ast_\beta(\mathcal A^\ast, d^\ast, \beta_\ast) \rto \dto^{^+\underline \Psi_k}   &H^\ast (\mathcal A^\ast, d^\ast)\dto^{^+\underline \Psi_k} \\
&^+H^\ast_\beta(\mathcal A^\ast, d^\ast, \beta_\ast)\rto    &H^\ast (\mathcal A^\ast, d^\ast)
\enddiagram
can be  derived by passing to cohomology in the  commutative  diagram of CDGA's.

\diagram
&(\mathcal A^\ast \otimes \Lambda[u], \mathcal D[u])\rto \dto^{\Psi_k[u]}   &(\mathcal A^\ast, d^\ast)\dto^{ \Psi_k} \\
&(\mathcal A^\ast \otimes \Lambda[u], \mathcal D[u])\rto  &(\mathcal A^\ast, d^\ast)
\enddiagram
where $\Lambda[u]$ is the free commutative graded algebra generated by the symbol  $u$ of degree 2,
$\mathcal D[u] (a\otimes u^r)= d(a)\otimes u^r + \beta (a)\otimes u^{r+1}$ and 
$\Psi [u] (a\otimes u^r)= 1/ k^r \Psi (a)\otimes u^r .$
\vskip .1in

For $P$ 1-connected and $(\Lambda([V], d_V) $ a minimal model of $(\Omega^\ast (P), d^\ast)$
we want to establish the existence of the homotopy commutative  diagram 
\vskip .1in 
\diagram
&&&&&A \rrto  \xto' [d] [dd]_{\Psi^P_k[u]} & & B\ddto_{\Psi^P_k}  \\
&&&&C \urto^{\tilde\theta} \rrto \ddto_{\Psi_k[u]} & & D\urto^\theta \ddto_{\Psi_k} \\
&&&&&A \xto' [r] [rr] & & B\\
&&&&C\rrto \urto^{\tilde\theta}  & &D \urto^\theta
\enddiagram

where 

\begin{equation*}
\begin{aligned}
A= &(\Omega_X(P^{S^1}) \otimes \Lambda[u], \mathcal D[u])\\
B= &(\Omega_X(P^{S^1}) , d)\\
C= &(\Lambda[V\oplus \overline V]\otimes \Lambda[u], \delta[u])\\
D= &(\Lambda[V\oplus \overline V], \delta)
\end{aligned}
\end{equation*}

with
\begin{equation*}
\begin{aligned}
 \mathcal D[u] (\omega\otimes u^r)= &d (\omega)\otimes u^r + i^X(\omega)\otimes u^{r+1}\\
\delta[u] (a \otimes u^r)= &\delta(a) \otimes u^r + i^V(\omega)\otimes u^{r+1}.
\end{aligned}
\end{equation*}

The existence of the quasi isomorphism $\theta$ was established in  in [SV].  The existence of the quasi isomorphism $\tilde\theta$ and the homotopy commutativity of the top square was established in [BV]
and the homotopy commutativity of the side squares was verified in [BFG].
The right side square resp. left side square in this diagram provide  identifications of $HH^\ast(\Lambda[V], d_V)$ with $hH^\ast(P)$ resp. of $CH^\ast(\Lambda[V], d_V)$ with $cH^\ast(P).$ These identifications are 
compatible with all natural transformations defined above and with the endomorphisms induced by the algebraic resp. geometric power maps.  In particular  one drive Theorem \ref{T3} from Theorem \ref{T4}. 
It is tedious but straightforward to derive, under the hypothesis of 1-- connectivity for $P,$ the
identification of the diagram Fig 6  for $P$ and the diagram Fig 7 for  $(\Omega^\ast(P), d^\ast).$

\end{document}